\documentclass[letter, 9pt]{article}
\usepackage{amsmath}
\usepackage{mathrsfs} 
\usepackage{amsthm}
\usepackage{amsfonts}
\usepackage{amssymb}
\usepackage{latexsym}
\usepackage{tikz} 
\usepackage{esint}
\usepackage{mathtools}
\usepackage{stmaryrd}

\usepackage{tikz-cd}
\usepackage{cancel}
\usepackage{sectsty}
\usepackage{hyperref}
\usepackage{caption}
\usepackage{float}
\usepackage{setspace}

\usepackage{comment}
\usepackage{quiver}

\usepackage{titlesec}
\titleformat{\subsection}{\normalfont\bfseries}{\thesubsection}{1em}{}
\titleformat{\subsubsection}[runin]{\normalfont\bfseries}{\thesubsubsection}{1em}{}

\renewcommand{\thesubsubsection}{(\thesubsection.\arabic{subsubsection})}

\usepackage{tocloft} 

\usepackage[OT2,T1]{fontenc}
\DeclareSymbolFont{cyrletters}{OT2}{wncyr}{m}{n}
\DeclareMathSymbol{\Sha}{\mathalpha}{cyrletters}{"58}

\sectionfont{\fontsize{12}{15}\selectfont}
\subsectionfont{\fontsize{11}{15}\selectfont}
\subsubsectionfont{\fontsize{10}{15}\selectfont}

\usepackage[colorinlistoftodos,textsize=scriptsize]{todonotes}
\usepackage{marginnote}

\newcommand{\mytodo}[2][]{{%
 \let\marginpar\marginnote
 \reversemarginpar
 \renewcommand{\baselinestretch}{0.8}%
 \todo[#1]{#2}}}

\usepackage[matrix,tips,graph,curve]{xy}
\usepackage{enumitem}

\makeatletter

\usepackage[
backend=biber,
style=numeric,
sorting=anyt, doi=false, url=false, isbn=false, giveninits=true, maxnames = 6,
]{biblatex}

\makeatother


\makeatletter

\theoremstyle{plain}
\newtheorem{theorem}[subsubsection]{Theorem}
\newtheorem{corollary}[subsubsection]{Corollary}
\newtheorem{lemma}[subsubsection]{Lemma}
\newtheorem{proposition}[subsubsection]{Proposition}

\newtheorem*{theoremA}{Theorem A}
\newtheorem*{theoremB}{Theorem B}

\theoremstyle{definition}
\newtheorem{definition}[subsubsection]{Definition}

\newtheorem{remark}[subsubsection]{Remark}

\newtheorem{set-up}[equation]{Set-up}

\newcommand{\IA}{\mathbb{A}}

\newcommand{\IC}{\mathbb{C}}

\newcommand{\IQ}{\mathbb{Q}}
\newcommand{\IR}{\mathbb{R}}
\newcommand{\IS}{\mathbb{S}}

\newcommand{\IW}{\mathbb{W}}

\newcommand{\IZ}{\mathbb{Z}}

\newcommand{\End}{\mathrm{End}}

\newcommand{\Hom}{\mathrm{Hom}}
\newcommand{\Aut}{\mathrm{Aut}}

\newcommand{\Spec}{\mathrm{Spec}}

\newcommand\iso{\,{\cong}\,} 
\newcommand\tensor{{\otimes}}

\newcommand{\into}{\hookrightarrow}

\def\d/{/\mspace{-6.0mu}/}

\def\wt{\widetilde}

\def\what{\widehat}

\newcommand{\sV}{\mathcal{V}}

\newcommand{\sF}{\mathcal{F}}
\newcommand{\sH}{\mathcal{H}}

\newcommand{\w}{\omega}

\newcommand{\Ohm}{\Omega}

\newcommand{\id}{\mathrm{id}}

\newcommand{\fp}{\mathfrak{p}}

\newcommand{\sM}{\mathcal{M}}
\newcommand{\sO}{\mathcal{O}}

\newcommand{\Gal}{\mathrm{Gal}}

\newcommand{\Hdg}{\mathrm{Hdg}}

\newcommand{\Isom}{\mathrm{Isom}}

\newcommand{\et}{\mathrm{{\acute{e}}t}}

\newcommand{\Sh}{\mathrm{Sh}}

\newcommand{\GSp}{\mathrm{GSp}}
\newcommand{\GL}{\mathrm{GL}}

\newcommand{\sA}{\mathcal{A}}

\newcommand{\bpi}{\mathbf{\pi}}
\newcommand{\dR}{\mathrm{dR}}

\newcommand{\Fil}{\mathrm{Fil}}

\newcommand{\Ab}{\mathsf{Ab}}

\renewcommand{\Spec}{\mathrm{Spec\,}}

\newcommand{\AH}{\mathrm{AH}}
\newcommand{\an}{\mathrm{an}}

\newcommand{\Mot}{\mathsf{Mot}}

\renewcommand{\bpi}{\boldsymbol{\pi}}

\renewcommand{\H}{\mathrm{H}}

\newcommand{\sto}{\stackrel{\sim}{\to}}

\newcommand{\sG}{\mathcal{G}}

\newcommand{\fh}{\mathfrak{h}}

\newcommand{\sfK}{\mathsf{K}}

\newcommand{\shM}{\mathscr{M}}

\newcommand{\bs}{\mathbf{s}}
\newcommand{\bt}{\mathbf{t}}

\newcommand{\Loc}{\mathsf{Loc}}

\newcommand{\bv}{\boldsymbol{v}}

\newcommand{\bc}{\mathrm{bc}}

\newcommand{\sfV}{\mathsf{V}}
\newcommand{\sfW}{\mathsf{W}}
\newcommand{\sfH}{\mathsf{H}}
\newcommand{\sfR}{\mathsf{R}}

\newcommand{\am}{\mathrm{am}}

\title{{\large{\textbf{A Note on Systems of Realizations on Shimura Varieties}}}
\author{\normalfont{Ziquan Yang}}
\vspace{-1ex}
  
\date{\vspace{-5ex}}}

\begin{document}

\maketitle

\begin{abstract}
    Let $(G, \Omega)$ be a Shimura datum of abelian type. It is well known that the corresponding Shimura variety $\mathrm{Sh}(G, \Omega)$ should be a moduli space of abelian motives equipped with some additional structures. In this half-expository note, we give under some simplifying assumptions a moduli interpretation of $\mathrm{Sh}(G, \Omega)$ over the reflex field purely in terms of systems of realizations. The main purpose is to introduce some convenient formalism that can be used to avoid the technicalities about dealing with various notions of families of motives. 
\end{abstract}

\tableofcontents

\section{Introduction}

It is well known that Shimura varieties of abelian type are supposed to be moduli spaces of ``abelian motives equipped with centain additional structures''. The purpose of this note is to provide a reference for a precise statement, under some simplifying assumptions. Let $(G, \Ohm)$ be a Shimura datum of abelian type with reflex field $E$, and for each compact open subgroup $\sfK \subseteq G(\IA_f)$, denote the Shimura variety $G(\IQ) \backslash \Ohm \times G(\IA_f)/ \sfK$ by $\Sh_\sfK(G, \Ohm)_\IC$ and its canonical model over $E$ by $\Sh_\sfK(G, \Ohm)$. Let $\Sh(G, \Ohm)$ be the limit $\varprojlim_\sfK \Sh_\sfK(G, \Ohm)$ as $\sfK$ runs through compact open subgroups. 

Before giving a moduli interpretation, we first show that: 

\begin{theoremA}
\emph{(see also \ref{thm: autSystem exists})}
Assume that $(G, \Ohm)$ has $\IQ$-rational weight, the center of $G$ has no anisotropic torus split over $\IR$, and $(G, \Ohm)$ has a nice Hodge-type cover (see \ref{def: nice cover}). Then there is a natural functor from the category $\mathrm{Rep}(G)$ of finite-dimensional $G$-representations over $\IQ$ to the category of $G(\IA_f)$-equivariant weakly abelian-motivic systems of realizations over $\Sh(G, \Ohm)$.
\end{theoremA}

Roughly speaking, a system of realizations is a cohomological avatar of a ``genuine'' family of motives. For a subfield $k \subseteq \IC$ and a smooth $k$-variety $S$, a system of realizations is given by a variation of Hodge structures (VHS) over $S_\IC$ whose de rham and \'etale components descend to $S$ (see \ref{def: system of realizations}). The advantage of working with systems of realizations, as opposed to ``genuine'' families of motives, is that the former is defined purely in terms of sheaves so that its pullback and descent properties are well understood (see also \ref{rmk: compare with motives}). The only price to pay is that we need to explicitly keep track of a compatability condition between the descent data of the \'etale and de Rham realizations of the VHS in a system. This is why we introduce the notion of \textit{weakly abelian-motivic} systems of realizations (see \ref{def: compatible with AM}), which are those whose descent data are controlled by abelian motives. It is a pointwise definition and may look a little odd at first glance, but we hope the reader will find it actually very flexible to work with once one gets used to it (see \ref{sec: motivation for weakly AM}). 

\begin{theoremB}
    Let $(G, \Ohm)$ be as in Theorem~A and $\sfK$ be a neat compact open subgroup. Then for every faithful $G$-representation $V$ and smooth $E$-variety $S$, there is a natural bijection between $\Sh_\sfK(G, \Ohm)(S)$ and isomorphism classes of pairs $(\sfW, [\xi])$ such that $\sfW$ is a weakly abelian-motivic system of realizations over $S$, and $[\xi]$ is a $\sfK$-level structure which is $\sfW$-rational and is of type $\Ohm$.
\end{theoremB}

A more precise version is stated in \ref{thm: moduli over reflex field}. The meaning of ``$\sfW$-rational'' and ``of type $\Ohm$'' are explained in \S\ref{sec: level str on systems}, and the motivation for these definitions is explained in \ref{rmk: rational level structure}. The moduli interpretation of the complex Shimura variety $\Sh_\sfK(G, \Ohm)_\IC$ in terms of VHS equipped with additional structures (\cite[Prop.~3.10]{Mil94}) follows relatively easily from the complex uniformization $G(\IQ) \backslash \Ohm \times G(\IA_f) / \sfK$ and standard results due to Baily and Borel. This note is mainly focused on the canonical model $\Sh_\sfK(G, \Ohm)$ over the reflex field. Morally of course Theorem B should follow from \cite[Thm~3.31]{Mil94}. Unfortunately, some proofs in \textit{loc. cit.} seem to be only sketched, and there are some technical subtleties about the notion of a family of motives used there that should be further explained (see also \ref{rmk: compare with motives}).

The author hopes that the notion of weakly abelian-motivic systems of realizations set up in this note gives a convenient alternative formalism for one to fill in the small gaps in literature. Our general approach here is inspired by \cite[\S2.2]{KisinInt} and \cite[\S5.24]{CSpin}. We expect that the simplifying assumption of the existence of a ``nice Hodge-type cover'' \ref{def: nice cover} can be removed by suitably adapting some constructions in \cite{KisinInt}, \cite{Lovering}, and \cite{Milne:CanonicalModels}, if needed. 

\paragraph{Acknowledgment} We thank Ben Moonen for encouraging the author to write this note and Paul Hamacher for his comments. 

\paragraph{Notations} 

Let $K$ be a field and $L$ be an extension of $K$. For any $K$-scheme $X$, we write $X(L)$ for the set of $K$-linear morphisms $\Spec(L) \to X$. 


By a variety over a field, we shall always mean a geometrically reduced separated scheme of finite type over the field. If $X$ is a smooth projective variety over $\IC$, we write $\H^i_B(X, \IQ)$ for the $i$th Betti (i.e., singular) cohomology and $\H^i(X, \IQ)$ for the associated Hodge structure.

Let $R$ be a commutative ring and $M$ be a finite free $R$-module. Let $M^\tensor$ denote the direct sum of all $R$-modules that can be formed from $M$ using the operations of taking duals, tensor products, symmetric powers and exterior powers. Whenever applicable, we use these notations also for sheaves of modules in some Grothendieck topology which is understood in the context. 



If $S$ is a smooth variety over $\IC$ and $\IW$ is a VHS (variation of Hodge structures) over $S$, we often use $\IW_B$ to denote the underlying $\IQ$-local system and $\IW_\dR$ the filtered flat vector bundle over $S$. We will always assume that the VHS is \textit{pure}, i.e., the weight filtration is split. 

Let $S$ be a scheme. We write $\mathsf{Loc}(S)$ for the category of \'etale local systems of $\IA_f$-modules and $\mathsf{Vect}(S)$ for the category of vector bundles on $S$. If $S$ is a smooth $\IC$-variety and $\sV \in \mathsf{Vect}(S)$, write $\sV^\an$ for the associated complex-analytic vector bundle. For a field-valued point $s : \Spec(\kappa) \to S$, and $\sF \in \Loc(S)$ (resp. $\sV \in \mathsf{Vect}(S)$), write $\sF_s$ for $s^* \sF$ (resp. $\sV_s$ for $s^* \sV = \sV \tensor_{\sO_S} \kappa$).

\section{Preliminaries}
\label{Sec: Prelim}
\subsection{Generalities on Galois Descent}
\label{sec: gen Gal Des}
We will make heavy use of Galois descent arguments, so we first spell out some tautology and set up some notations to facilitate discussion. We do not intend to state the results in utmost generality, but only the simple cases we need. 

Let $k$ be a field of characteristic $0$ and $L$ is an algebraically closed field over $k$. We allow $L$ to have infinite transcendence degree over $k$, e.g., $L = \IC$ and $k$ is a number field. Let $X_0$ be a smooth $k$-variety. Let $X$ be the base change $X_0 \tensor_k L$. Note that there is a canonical identification $X(L) = X_0(L)$. Given a $k$-linear morphism $s : \Spec(L) \to X_0$ (i.e., an element of $X_0(L)$) and $\sigma \in \Aut(L/k)$, let $\sigma(s)$ be the composition $\Spec(L) \stackrel{\Spec(\sigma)}{\to} \Spec(L) \stackrel{s}{\to} X_0$. Then $s \mapsto \sigma(s)$ defines the $\Aut(L/k)$-action on $X_0(L)$, and thus on $X(L)$ via the identificaton $X(L) = X_0(L)$ above. 


\subsubsection{} \label{sec: sheaves descent iso} Suppose now that $\sF \in \Loc(X)$. A descent of $\sF$ to $X_0$ is a pair $(\sF_0 \in \Loc(X_0), \alpha : \sF_0|_{X} \iso \sF)$. For every $x \in X(L)$ and $\sigma \in \Aut(L/k)$, the descent $(\sF_0, \alpha)$ gives us an isomorphism of stalks $\sF_{x} \sto \sF_{\sigma(x)}$, which we label as $\sigma_{\sF_0, x}$. Indeed, when we write $\sF_x$ and $\sF_{\sigma(x)}$, we are viewing $x$ as an element of $X(L)$, but via the identification $X(L) = X_0(L)$ we also have identifications $\sF_x = \sF_{0, x}$ and $\sF_{\sigma(x)} = \sF_{0, \sigma(x)}$. By applying \cite[04FN]{stacks-project} to $(X_0, x, \sigma(x))$, we obtain a canonical isomorphism $\sF_{0, x} \sto \sF_{0, \sigma(x)}$. This gives $\sigma_{\sF_0, x}$. If $x_0 \in X_0$ is the image of the geometric point $x$, then $x = \sigma(x)$ for $\sigma \in \Aut(k(x)/k(x_0)) = \Aut(L/k(x_0))$. The collection $\{ \sigma_{\sF_0, x} \}_{\Aut(L/k(x_0))}$ gives precisely the $\Aut(L/k(x_0))$-action on the stalk $\sF_{0, x}$ (\cite[03QW]{stacks-project}).

Similarly, for $\sV \in \mathsf{Vect}(X)$, a descent of $\sV$ to $X_0$ is a pair $(\sV_0 \in \mathsf{Vect}(X_0), \beta : \sV_0 |_X \iso \sV)$, which gives at each $x \in X(L)$ a $\sigma$-linear isomorphism $\sigma_{\sV_0, x} : \sV_x \sto \sV_{\sigma(x)}$. A global section $\bs \in \H^0(X, \sF)$ (resp. $\bt \in \H^0(X, \sV)$) descends to $\H^0(X_0, \sF_0)$ (resp. $\H^0(X_0, \sV_0)$) if and only if for every $\sigma \in \Aut(L/k)$ and $x \in X(L)$, $\sigma_{\sF_0, x}(\bs_x) = \bs_{\sigma(x)}$ (resp. $\sigma_{\sV_0, x}(\bt_x) = \bt_{\sigma(x)}$). 


We often omit the identification $\alpha$ from the notation, and simply say that $\sF_0$ is a descent of $\sF$. However, we emphasize that as an object of $\Loc(X_0)$, $\sF_0$ may certainly have nontrivial automorphisms, but there is no nontrivial automorphism, say $g \in \Aut(\sF_0)$, such that $\alpha \circ (g|_{X} : \sF_0|_X \sto \sF_0|_X) = \alpha$. That is, as a descent of $\sF$, $\sF_0$ cannot have nontrivial automorphisms. A similar remark applies to descent of vector bundles. Therefore, if we say that two descents of a fixed object in $\Loc(X)$ or $\mathsf{Vect}(X)$ to $X_0$ are isomorphic, then they are \textit{uniquely isomorphic}. 

\subsubsection{} \label{sec: simultaneous descent} Let $X, X_0, \sF, (\sF_0, \alpha : \sF_0|_X \iso \sF)$ be as above. Suppose that $X_0$ is a finite \'etale cover of another smooth $k$-variety $Y_0$ with $Y := Y_0 \tensor_k L$, and $\sF$ admits a descent $(\sG \in \Loc(Y), \gamma : \sG|_{X} \iso \sF)$. A \textbf{simultaneous descent} of $(\sF_0, \sG)$ to $Y_0$ is a triple $(\sG_0 \in \Loc(Y_0), \xi : \sG_0|_{X_0} \iso \sF_0, \zeta : \sG_0|_Y \iso \sG)$ such that the diagram on the right below commutes: 

\[\begin{tikzcd}
	X & {X_0} & \sF & {\sF_0|_X} \\
	Y & {Y_0} & {\sG|_X} & {\sG_0|_X}
	\arrow[from=1-1, to=1-2]
	\arrow[from=1-1, to=2-1]
	\arrow[from=2-1, to=2-2]
	\arrow[from=1-2, to=2-2]
	\arrow["\lrcorner"{anchor=center, pos=0.125}, draw=none, from=1-1, to=2-2]
	\arrow["\alpha"', from=1-4, to=1-3]
	\arrow["\gamma", from=2-3, to=1-3]
	\arrow["{\zeta|_X}"', from=2-4, to=2-3]
	\arrow["{\xi|_X}"', from=2-4, to=1-4]
\end{tikzcd}\]

It is not hard to see that such a triple $(\sG_0, \xi, \zeta)$ is \textit{unique up to unique isomorphism} provided that it exists. Similarly, we define the notion of simultaneous descent of vector bundles. 

\subsection{Motives and Absolute Hodge Cycles}
Let $k$ be a field of characteristic $0$. We denote by $\Mot_\AH(k)$ the neutral $\IQ$-linear Tannakian category of motives over $k$ for absolute Hodge cycles (cf. \cite[\S2]{Pan94} where it is denoted by $\shM_k$). We have the Tate objects $\mathbf{1}(n)$ for every $n \in \IZ$ in this category. For any object $M \in \Mot_\AH(k)$, we write $M(n)$ for $M \tensor \mathbf{1}(n)$ and by an absolute Hodge class on $M$ we mean a morphism $\mathbf{1} \to M$.




Following \cite[\S1]{MPTate}, we denote by $\w_\ell$ the $\ell$-adic realization functor which sends $\Mot_\AH(k)$ to the category of finite dimensional $\IQ_\ell$-vector spaces with an action of $\Gal_k := \Gal(\bar{k}/k)$, where $\bar{k}$ is some chosen algebraic closure of $k$. Putting $\w_\ell$'s together we obtain $\w_\et$, which takes values in the category of finite free $\IA_f$-modules with a $\Gal_k$-action. Let $\w_\dR$ denote the de Rham realization functor, which takes values in the category of filtered $k$-vector spaces. If $k = \IC$, we additionally consider the Betti realization $\w_B$ (resp. the Hodge realization $\w_\Hdg$) which takes values in the category of $\IQ$-vector spaces (resp. Hodge structures). 


Since the operations of taking duals, tensor products, and symmetric and exterior powers make sense for the Tannakian category $\Mot_\AH(k)$ and are compatible with the cohomological realizations, it makes sense to say whether a tensor $s \in (\w_\et(M) \times \w_\dR(M))^\tensor$ (resp. $s \in \w_B(M)^\tensor$) is absolute Hodge for $M \in \Mot_\AH(k)$ (resp. $M \in \Mot_\AH(\IC)$, when $k = \IC$).

Let $\Mot_\Ab(k) \subseteq \Mot_\AH(k)$ be the full Tannakian sub-category generated by the Artin motives and the motives attached to abelian varieties.
As in \textit{loc. cit.}, we will repeatedly make use of the following fact: 

\begin{theorem}
\emph{(\cite[Ch~1]{LNM900})}
\label{thm: tensor on AV always AH}
     The functor $\w_\Hdg$ is fully faithful when restricted to $\Mot_\Ab(\IC)$. In particular, for every $M \in \Mot_\Ab(\IC)$, every element $s \in \w_B(M) \cap \Fil^0 \w_\dR(M)$ is given by an absolute Hodge cycle.  
\end{theorem}

We often refer to objects in $\Mot_\Ab(k)$ as abelian motives. 

\subsubsection{} \label{sec: twist tensors} If $L$ is an extension of $k$, then there is a natural faithful functor of Tannakian categories $- \tensor_k L : \Mot_\AH(k) \to \Mot_\AH(L)$ compatible with fiber functors. A construction we shall often consider is the following: Let $M$ be an object of $\Mot_\AH(\IC)$ and $\sigma \in \Aut(\IC)$. We denote $M \tensor_\sigma \IC$ by $M^\sigma$. Base change properties of \'etale and de Rham cohomology give us canonical isomorphisms $\w_\et(M) \iso_\bc \w_\et(M^\sigma)$ and $\w_\dR(M) \iso_\bc \w_\dR(M^\sigma)$. Note that the latter is $\sigma$-linear. Here the supscript ``bc'' stands for ``base change'', and we add this supscript when we need to distinguish these isomorphisms from other potentially different isomorhisms (e.g., in most Galois descent arguments). If $s \in \w_B(M)^\tensor$ is absolute Hodge, then we write $s^\sigma$ for the (necessarily unique) element in $\w_B(M^\sigma)^\tensor$ such that $s$ and $s^\sigma$ have the same \'etale and de Rham realizations via the isomorphisms $\w_\et(M) \iso_\bc \w_\et(M^\sigma)$ and $\w_\dR(M) \iso_\bc \w_\dR(M^\sigma)$. 


\subsection{Tensors and Level Structures}

\subsubsection{} \label{sec: level structure bc}
Let $G$ be a connected reductive group over $\IQ$, $V$ be a finite dimensional $G$-representation over $\IQ$ and $\sfK \subseteq G(\IA_f)$ be a compact open subgroup. Let $T$ be a Noetherian $\IQ$-scheme and $\sfW_\et \in \Loc(T)$. A \textbf{$(G, V, \sfK)$-level structure}, or simply a $\sfK$-level structure when $(G, V)$ is understood, on $\sfW_\et$ is the data $[\eta]$ that give each geometric point $t \to T$ a $\pi_1^\et(T, t)$-invariant element $[\eta]_t \in \sfK \backslash \mathrm{Isom}(V \tensor \IA_f, \sfW_{\et, t})$, such that if $t, t'$ are two geometric points on the same connected component of $T$, parallel transport along any \'etale path from $t$ to $t'$ sends $[\eta]_t$ to $[\eta]_{t'}$. Alternatively, one can also consider the pro-\'etale sheaf $\underline{\Isom}(V\tensor \IA_f, \sfW_\et)$ consisting of isomorphims $V\tensor \IA_f \sto \sfW_\et$ over some pro-\'etale cover and define a $\sfK$-level structure to be a global section of $\sfK \backslash \underline{\Isom}(V\tensor \IA_f, \sfW_\et)$.\footnote{When the action of $\sfK$ on $V \tensor \IA_f$ is not faithful, this quotient is interpreted as taking the orbit of the action by the image of $\sfK$ in $\GL(V \tensor \IA_f)$. We are not considering a ``stacky'' quotient.}

If $T' \to T$ is a morphism of Noetherian $\IQ$-schemes, we naturally construct the \textbf{pullback} of $[\eta]$ to $T'$, which is a $\sfK$-level structure on $\sfW_\et|_{T'}$. We denote it by $[\eta]_{T'}$. If $M \in \Mot_\AH(\IC)$ is a motive, and $[\eta_M]$ is a $\sfK$-level structure on $\w_\et(M)$, then $[\eta_M]^\sigma$ denotes $\sfK$-level structure on $\w_\et(M^\sigma)$ induced by composing (a representative of) $[\eta_M]$ with $\w_\et(M) \iso_\bc \w_\et(M^\sigma)$. 

\subsubsection{} \label{sec: define eta(v)} We note that if a $\sfK$-level structure $[\eta]$ can be put on $\sfW_\et$, then for each $v \in (V^\tensor)^G$ (i.e., the $G$-invariants in $V^\tensor$), there exists a unique section $\bv_\et \in \H^0(\sfW_\et^\tensor)$ such that for every geometric point $t \to T$, some (and hence every) representative of $[\eta]_t$ sends $v \tensor 1$ to $\bv_{\et, t}$. We write this global section as $\eta(v)_\et$, since it is completely determined by $\eta$ and $v$. 

We will often make use of the following simple observation: Let $k$ and $L$ be as in (\ref{sec: gen Gal Des}), and suppose that $T$ is a smooth $k$-variety. A $\sfK$-level structure $[\eta]$ on $\sfW_\et |_{T_L}$ descends to a level structure on $\sfW_\et$ if and only if for every $t \in T(L) = T_L(L)$ and $\sigma \in \Aut(L/k)$, $\sigma_{\sfW_\et, t}$ sends $[\eta]_t$ to $[\eta]_{\sigma(t)}$. 

\section{Systems of Realizations}
\subsection{Basic Definitions}
\begin{definition}
\label{def: system of realizations}
    Let $k$ be a subfield of $\IC$ and $S$ be a smooth $k$-variety. By a system of realizations we mean a tuple $\sfV = (\sfV_B, \sfV_\dR, \sfV_\et, i_\dR, i_\et)$ where 
    \begin{itemize}
        \item $\sfV_B$ is a $\IQ$-local system over $S_\IC := S \tensor_k \IC$; 
        \item $\sfV_\dR$ is a filtered flat vector bundle over $S$; 
        \item $\sfV_\et$ is an \'etale local system of $\IA_f$-coefficients over $S$; 
        \item $i_\dR$ is an isomorphism of flat vector bundles $(\sfV_B \tensor \sO^\an_{S_\IC}, \mathrm{id} \tensor d) \sto (\sfV_\dR|_{S_\IC})^\an$ such that $(\sfV_B, \sfV_\dR|_{S_\IC})$ forms a pure \textit{polarizable} VHS;
        \item $i_\et$ is an isomorphism $\sfV_B  \tensor \IA_f \sto \sfV_\et|_{S_\IC}$.
    \end{itemize}  
\end{definition}
We may often omit $(i_\dR, i_\et)$ in the notation and simply write $\sfV_B \tensor \IA_f = \sfV_\et|_{S_\IC}$ and $\sfV_B \tensor \sO_{S_\IC}^\an = \sfV_\dR|_{S_\IC}$. There are some changes of topology here: Technically, $\sfV_B \tensor \IA_f$ is defined for the analytic topology, but it has a canonically associated \'etale local system, which we also denote by $\sfV_B \tensor \IA_f$. Similarly, we often do not distinguish $\sfV_\dR|_{S_\IC}$ with its analytification $\sfV_B \tensor \sO^\an_{S_\IC}$, because over a smooth complex variety there is a natural equivalence of categories between flat algebraic vector bundles with regular singularieties and flat holomorphic vector bundles (\cite[Thm~II.5.9]{DelVB}). In particular, a flat global section of $\sfV_B \tensor \sO_{S_\IC}^\an$ is also a flat global section of $\sfV_{\dR}|_{S_\IC}$. 

\subsubsection{} \label{sec: pullback systems} If $k'$ is another subfield of $\IC$ which contains $k$, $S'$ is a $k'$-variety, and there is a $k$-linear morphism of schemes $S' \to S$, then we can naturally form the restriction $\sfV|_{S'}$ (or pullback) of $\sfV \in \sfR(S)$ to $S'$. Indeed, $S' \tensor_{k'} \IC$ is the same as $S \tensor_k \IC$, so the Betti component $\sfV_B$ stays the same. Then we simply restrict $\sfV_\dR$ and $\sfV_\et$ to $S'$. If $k = \IC$, we do not distinguish $\sfR(S)$ with the category of polarizable VHS, as these two categories the naturally equivalent. 

\subsubsection{} \label{sec: compare definitions}
Our definition looks slightly different from \cite[Def.~6.1]{FuMoonen} (see also \cite[\S3.41]{Mil94}). However, if $K$ is a number field (which is no longer equipped with a chosen embedding into $\IC$) and $S$ is a smooth $K$-variety, then a system of realizations over $S$ as defined in \cite[Def.~6.1]{FuMoonen} is equivalent to a system of realizations over $S_{(\IQ)}$ in our definition, where $S_{(\IQ)}$ is the $\IQ$-variety obtained by composing the structural morphism of $S$ with $\Spec(K) \to \Spec(\IQ)$.  

\subsubsection{} \label{sec: global section of realizations} With the obvious definitions of morphisms, tensor products and duals, systems of realizations over a smooth variety $S$ over $k \subseteq \IC$ form a $\IQ$-linear Tannakian category $\sfR(S)$. We have Tate objects $\mathbf{1}(n)$ and denote $\sfV(n)$ by $\sfV \tensor \mathbf{1}(n)$ (cf. the discussion below \textit{loc. cit.}). We write $\H^0(\sfV)$ for the $\IQ$-module $\Hom(\mathbf{1}, \sfV)$. Concretely, an element in $\H^0(\sfV)$ is given by a triple $(\bs_B, \bs_\dR, \bs_\et)$ where $\bs_B \in \H^0(\sfV_B)$ is a global section which is everywhere of Hodge type $(0, 0)$, and $\bs_{\dR}$ and $\bs_\et$ are respectively global sections in $\H^0(\sfV_\dR)$ and $\H^0(\sfV_\et)$ whose restrictions over $S_\IC$ are the de Rham and \'etale realizations of $\bs_B$. 



\begin{remark}
\label{rmk: compare with motives}
    We introduce systems of realizations mainly to circumvent the problem of not having a definition of a family of motives in literature that has been documented to have satisfactory pullback and descent properties. For example, if one uses \cite[Def.~2.37]{Mil94}, then it is not a priori clear how to perform pullback along non-dominant morphisms, though when we restrict to considering abelian motives over smooth $\IC$-varieties this definition has nice descent properties (e.g., combine Lem.~2.33 and Prop.~2.42 in \textit{loc. cit.}). On the other hand, \cite[Def.~4.3.3]{Moonen-Fom} is amenable to defining pullback, but its descent properties have not been discussed.
\end{remark}


\subsection{Characterizing ``abelian-motivic'' systems}

\begin{definition}
\label{def: strongly AM}
    Let $S$ be a smooth variety over $k \subseteq \IC$. We say that a system of realizations $\sfV$ over $S$ is \textbf{strongly abelian motivic (strongly AM)} if there exists another system $\sfH$ which is given by the first cohomology of an abelian scheme $\sA / S$, a realization $\sfH^\natural$ that is a direct sum of those constructed out of $\sfH$ by taking Tate twists, duals, tensor products, exterior and symmetric powers, as well as an idempotent projector $\bpi = (\bpi_B, \bpi_\dR, \bpi_\et) : \sfH^\natural \to \sfH^\natural$ such that $\sfV$ is isomorphic to the image of $\bpi$. We write the full subcategory of $\sfR(S)$ consisting of such objects as $\sfR^\circ_{\am}(S)$. 
    
    We say that $\sfV$ is \textbf{\'etale-locally} strongly AM, if for some finite \'etale cover $S' \to S$ of $k$-varieties, the restriction $\sfV|_{S'}$ is strongly AM. We write the full subcategory of $\sfR(S)$ consisting of such objects as $\sfR_{\am}(S)$. 
\end{definition}

\begin{remark}
\label{rmk: strongly AM}
      In \cite{Mil94}, there is also a similar notion of an abelian motivic VHS (Def.~2.30), and a notion of an abelian motive on a smooth connected variety over a field of characteristic $0$ (Def.~2.37). To distinguish with Milne's definitions, we added ``strongly'' to indicate that we ask the relevant abelian scheme $\sA$ to be defined over the entire $S$, as opposed to just an open dense subset of $S$. One reason for doing so is that being strongly AM is clearly stable under base change along any morphisms between smooth $k$-varieties, rather than just the dominant ones (see also \ref{rmk: compare with motives}). 
\end{remark}


The following definition is an extension of \ref{def: strongly AM}. 

\begin{definition}
\label{def: compatible with AM}
     We say that $\sfV$ is \textbf{weakly abelian-motivic (weakly AM)} if for every $s \in S(\IC)$, there exists $M \in \Mot_\Ab(\IC)$ such that $\w_{\Hdg}(M) \iso (\sfV_{B, s}, \sfV_{\dR, s})$; moreover, for any such isomorphism $\gamma : \w_{\Hdg}(M) \sto (\sfV_{B, s}, \sfV_{\dR, s})$ and $\sigma \in \Aut(\IC/k)$, there exists an isomorphism $\gamma^\sigma : \w_\Hdg(M^\sigma) \sto (\sfV_{B, \sigma(s)}, \sfV_{\dR, \sigma(s)})$ such that the Betti components $\gamma_{B}, \gamma^\sigma_{B}$ of $\gamma$ and $\gamma^\sigma$ fit into commutative diagrams:
    \begin{equation}
    \label{diag: defining AM compatible}
        \begin{tikzcd}
	{\w_{\et}(M) } & {\sfV_{B,s} \tensor \IA_f} \\
	{\w_{\et}(M^\sigma)} & {\sfV_{B, \sigma(s)} \tensor \IA_f}
	\arrow["{\iso_{\mathrm{bc}}}"', from=1-1, to=2-1]
	\arrow["{\sigma_{\sfV_\et, s}}", from=1-2, to=2-2]
	\arrow["{\gamma^\sigma_{B} \tensor \IA_f}", from=2-1, to=2-2]
	\arrow["{\gamma_{B} \tensor \IA_f}", from=1-1, to=1-2]
    \end{tikzcd}
    \text{ and }
    \begin{tikzcd}
	{\w_{\dR}(M) } & {\sfV_{B,s} \tensor \IC} \\
	{\w_{\dR}(M^\sigma)} & {\sfV_{B, \sigma(s)} \tensor \IC}
	\arrow["{\iso_{\mathrm{bc}}}"', from=1-1, to=2-1]
	\arrow["{\sigma_{\sfV_\dR, s}}", from=1-2, to=2-2]
	\arrow["{\gamma^\sigma_B \tensor \IC}", from=2-1, to=2-2]
	\arrow["{\gamma_B \tensor \IC}", from=1-1, to=1-2]
    \end{tikzcd}
    \end{equation} We note that $\gamma^\sigma$ is uniquely determined by $\gamma$ provided that it exists. For the notations $\sigma_{\sfV_\et, s}$, $\sigma_{\sfV_\dR, s}$ and ``$\iso_\bc$'' in the above diagrams, see \ref{sec: sheaves descent iso} and \ref{sec: twist tensors}. When we write $\sfV_{B, s}$ and $\sfV_{B, \sigma(s)}$, we are viewing $s$ and $\sigma(s)$ as (usually different) closed points on $S_\IC$, and to apply the notations $\sigma_{\sfV_\et, s}$ and $\sigma_{\sfV_\dR, s}$ we are viewing $\sfV_\et$ and $\sfV_\dR$ as descents of $\sfV_B \tensor \IA_f$ and $\sfV_B \tensor \sO^{\mathrm{an}}_S$ over $S$ respectively. 
    
    Denote the full subcategory of $\sfR(S)$ given by these objects by $\sfR^*_\am(S)$. 
\end{definition}

In the above definition, we say that the pair $(M, \gamma)$ is a ``test object'' at $s$ for $\sfV$ to be weakly AM, but its choice does not matter: 
\begin{lemma}
For a $\sfV \in \sfR(S)$ as in \ref{def: compatible with AM}, if the existence of $\gamma^\sigma$ is satisfied by one $(M, \gamma)$, then it is satisfied by any other pair $(M', \gamma')$.
\end{lemma}
\begin{proof}
    Consider the isomorphism $\psi:= (\gamma')^{-1} \circ \gamma : \w_\Hdg(M) \sto \w_\Hdg(M')$. Since $M$ and $M'$ are both abelian motives (i.e., objects on $\Mot_\Ab(\IC)$), $\psi$ is absolute Hodge. Then for any $\sigma \in \Aut(\IC/k)$, $\psi$ induces an isomorphism $\psi^\sigma : \w_B(M^\sigma) \sto \w_B((M')^\sigma)$ such that for the Betti components $\psi_B$ and $\psi^\sigma_B$ the diagrams 
    \begin{equation}
     \begin{tikzcd}
	{\w_\et(M) } & {\w_{\et}(M')} \\
	{\w_\et(M^\sigma)} & {\w_{\et}((M')^\sigma)}
	\arrow["{\iso_{\mathrm{bc}}}"', from=1-1, to=2-1]
	\arrow["{\iso_{\mathrm{bc}}}", from=1-2, to=2-2]
	\arrow["{\psi_B^\sigma \tensor \IA_f}", from=2-1, to=2-2]
	\arrow["{\psi_B \tensor \IA_f}", from=1-1, to=1-2]
    \end{tikzcd}
    \text{ and }
    \begin{tikzcd}
	{\w_{\dR}(M)} & {\w_{\dR}(M')} \\
	{\w_{\dR}(M^\sigma)} & {\w_{\dR}((M')^\sigma)}
        \arrow["{\iso_{\mathrm{bc}}}"', from=1-1, to=2-1]
	\arrow["{\iso_{\mathrm{bc}}}", from=1-2, to=2-2]
        \arrow["{\psi_B^\sigma \tensor \IC}", from=2-1, to=2-2]
	\arrow["{\psi_B \tensor \IC}", from=1-1, to=1-2]
    \end{tikzcd}
    \end{equation} 
    commute. Therefore, the diagrams in (\ref{diag: defining AM compatible}) commute with $\gamma^\sigma$ replaced by $\gamma^\sigma \circ (\psi^\sigma)^{-1}$.
\end{proof}

\subsubsection{} \label{sec: motivation for weakly AM} Our motivation for defining $\sfR^*_\am(S)$ is that it is very flexible to work with. We list some simple observations which we will use repeatedly later: 

\begin{enumerate}[label=\upshape{(\roman*)}]
    \item $\sfR^*_\am(S)$ is a Tannakian subcategory of $\sfR(S)$. In particular, it is stable under taking Tate twists, tensor products, duals, symmetric and exterior products.
    \item Let $S'$ be another smooth $k$-variety admitting a morphism to $S$. If $\sfV \in \sfR^*_\am(S)$, then $\sfV|_{S'} \in \sfR^*_\am(S')$; if $S'(\IC) \to S(\IC)$ is surjective, then the converse is also true.
    \item $\sfR_\am(S)$ is contained in $\sfR^*_\am(S)$ (see \ref{prop: strong implies weak} below). 
\end{enumerate}

For readers' convenience, we spell out a  tautology on how the condition of being weakly abelian motivic interacts with global tensors and level structures: 

\subsubsection{} \label{sec: descent of one component} Let $S$ be as above and take $\sfV = (\sfV_B, \sfV_\dR, \sfV_\et) \in \sfR_\am^*(S)$. Let $\bv_B \in \H^0(\sfV^\tensor_B)$ be a global section which is everywhere of Hodge type $(0, 0)$. Let $\bv_{\et, \IC} \in \H^0(\sfV^\tensor_\et|_{S_\IC})$ and $\bv_{\dR, \IC} \in \H^0(\sfV^\tensor_\dR|_{S_\IC})$ be the \'etale and de Rham realizations of $\bv_B$ respectively. Let $s \in S(\IC)$, $\sigma \in \Aut(\IC/k)$ and $(M, \gamma)$ be a test object at $s$ for $\sfV$ to be weakly AM and remember the resulting $\gamma^\sigma$. Let $v_M := \gamma_{B}^{-1}(\bv_{B, s}) \in \w_B(M)^\tensor$. Then since $M$ is an abelian motive, $v_M$ is absolute Hodge. Therefore, it makes sense to form $v_M^\sigma$ (see \ref{sec: twist tensors}). By \eqref{diag: defining AM compatible} the following are equivalent: 
    \begin{enumerate}[label=\upshape{(\roman*)}]
        \item $\gamma^\sigma$ sends $v_M^\sigma$ to $\bv_{B, \sigma(s)}$. 
        \item $\gamma_B^\sigma \tensor \IA_f$ sends $v_M^\sigma \tensor 1$ to $\bv_{B, \sigma(s)} \tensor 1 \in \sfV^\tensor_{\et, \sigma(s)} = \sfV_{B, \sigma(s)}^\tensor \tensor \IA_f$. 
        \item $\gamma_B^\sigma \tensor \IC$ sends $v_M^\sigma \tensor 1$ to $\bv_{B, \sigma(s)} \tensor 1 \in \sfV^\tensor_{\dR, \sigma(s)} = \sfV^\tensor_{B, \sigma(s)} \tensor \IC$. 
    \end{enumerate}
    Using the tautological remarks in \ref{sec: sheaves descent iso}, one easily checks that (ii) (resp. (iii)) holds for every $s$ and $\sigma$ if and only if $\bv_{\et, \IC}$ (resp. $\bv_{\dR, \IC}$) descends to $S$. Therefore, using the equivalence of (ii) and (iii) with (i), we see that if one of $\bv_{\et, \IC}$ or $\bv_{\dR, \IC}$ descends to $S$, then so does the other. 
    
    Similarly, suppose that for some $(G, V, \sfK)$ in \ref{sec: level structure bc} we can put a $\sfK$-level structure $[\eta]$ on $\sfV_\et$. Again consider the $(M, \gamma, \gamma^\sigma)$ above. Let $[\eta_M]$ be the $\sfK$-level structure on $\w_\et(M)$ obtained by pulling back of $[\eta]_s$ along $\gamma_B \tensor \IA_f$. Then \ref{sec: define eta(v)} and the commutativity of the first diagram in \ref{def: compatible with AM} implies that $\gamma_B^\sigma \tensor \IA_f$ sends $[\eta_M]^\sigma$ to $[\eta]_{\sigma(s)}$, where $[\eta_M]^\sigma$ is the composition $V \tensor \IA_f \to \w_\et(M) \iso_\bc \w_\et(M^\sigma)$.\\

\indent Now we return to show \ref{sec: motivation for weakly AM}(iii): 

\begin{proposition}
\label{prop: strong implies weak}
    Let $S$ be a smooth variety over a field $k \subseteq \IC$ and $\sfV$ be an object of $\sfR(S)$. If $\sfV \in \sfR_\am(S)$, then $\sfV \in \sfR^*_\am(S)$. 
\end{proposition}
\begin{proof}
    First, we are allowed to assume that $\sfV$ is strongly abelian motivic (not just \'etale locally) thanks to \ref{sec: motivation for weakly AM}(ii). Let $\sA/S, \sfH, \sfH^\natural, \bpi$ be as in \ref{def: strongly AM}. Fix an arbitrary $s \in S(\IC)$ and $\sigma \in \Aut(\IC/k)$. Then the descent of the family $\sA_\IC / S_\IC$ to $\sA/S$ provides us with a canonical isomorphism $(\sA_{s})^\sigma := \sA_s \tensor_\sigma \IC \sto \sA_{\sigma(s)}$ of $\IC$-varieties. In analogy to \ref{sec: sheaves descent iso}, we denote it by $\sigma_{\sA, s}$. It tautologically fits into commutative diagrams 
    \begin{equation}
    \label{diag: AV tautology}
    \begin{tikzcd}
	{\H^1_\et(\sA_s, \IA_f)} & {\sfH_{\et, s}} \\
	{\H^1_\et((\sA_s)^\sigma, \IA_f)} & {\H^1_\et(\sA_{\sigma(s)}, \IA_f) = \sfH_{\et, \sigma(s)}}
	\arrow["{\iso_{\mathrm{bc}}}"', from=1-1, to=2-1]
	\arrow["\sigma_{\sA, s}", from=2-1, to=2-2]
	\arrow[equal, from=1-1, to=1-2]
	\arrow["{\sigma_{\sfH_\et, s}}", from=1-2, to=2-2]
    \end{tikzcd}
    \begin{tikzcd}
	{\H^1_\dR(\sA_s)} & {\sfH_{\dR, s}} \\
	{\H^1_\dR((\sA_s)^\sigma / \IC)} & {\H^1_\dR(\sA_{\sigma(s)}/\IC) = \sfH_{\dR, \sigma(s)}}
	\arrow["{\iso_{\mathrm{bc}}}"', from=1-1, to=2-1]
	\arrow["\sigma_{\sA, s}", from=2-1, to=2-2]
	\arrow[equal, from=1-1, to=1-2]
	\arrow["{\sigma_{\sfH_\dR, s}}", from=1-2, to=2-2]
    \end{tikzcd}
    \end{equation}
    We caution the reader that for the diagram on the right, the vertical arrows are $\sigma$-linear, whereas the horizontal ones are $\IC$-linear. 

     Fix an identification of $\sfV$ with $\mathrm{im}(\bpi)$ and denote by $T$ the tensorial construction such that $\sfH^\natural = T(\sfH)$. For every $x \in S(\IC)$, let $\fh^1(\sA_x) \in \Mot_\AH(\IC)$ be the motive whose Betti, \'etale and de Rham realisations are equal to the first Betti, \'etale and de Rham cohomology of $\sA_x$, respectively, and let $\fh^1(\sA_x)^\natural \coloneqq T(\fh^1(\sA_x))$. We denote by $M_x$ the submotive of $\fh^1(\sA_x)^\natural$ cut out by the absolute Hodge cycle $\bpi_{B, x}$, which is an idempotent projector on $\fh^1(\sA_x)^\natural$. Then we may identify $\w_\Hdg(M_x)$ with $(\sfV_{B, x}, \sfV_{\dR, x})$.
    
    
    We take $M = M_s$ and denote the identification $\w_\Hdg(M) = (\sfV_{B, s}, \sfV_{\dR, s})$ by $\gamma$. Then we use $(M, \gamma)$ as a test object at $s$ for $\sfV$ to be weakly AM. Note that $M^\sigma$ is canonically identified with the submotive of $T(\fh^1((\sA_s)^\sigma)) = T(\fh^1(\sA_x))^\sigma$ cut out by $\bpi^\sigma_{B, s}$. Now, since $\bpi_\et$ and $\bpi_\dR$ are global sections of $\End(\sfH^\natural_\et)$ and $\End(\sfH^\natural_\dR)$ respectively, $\sigma_{\sfH_\et, s}$ sends $\bpi_{\et, s}$ to $\bpi_{\et, \sigma(s)}$ and $\sigma_{\sfH_\dR, s}$ sends $\bpi_{\dR, s}$ to $\bpi_{\dR, \sigma(s)}$ (cf. \ref{sec: sheaves descent iso}). The diagrams in (\ref{diag: AV tautology}) hence tell us that $\sigma_{\sA, s}$ sends $\bpi^\sigma_{B, s}$ to $\bpi_{B, \sigma(s)}$. Therefore, we conclude that $\sigma_{\sA, s}$ induces an isomorphism $M^\sigma = (M_s)^\sigma \sto M_{\sigma(s)}$ of objects in $\Mot_\Ab(\IC)$. Denote this isomorphism by $\sigma_{M, s}$. Then the diagrams (\ref{diag: AV tautology}) induce commutative diagrams 
    \begin{equation}
    \label{diag: M tautology}
    \begin{tikzcd}
	{\w_\et(M)} & {\sfV_{\et, s}} \\
	{\w_\et(M^\sigma)} & {\w_\et(M_{\sigma(s)}) = \sfV_{\et, \sigma(s)}}
	\arrow["{\iso_{\mathrm{bc}}}"', from=1-1, to=2-1]
	\arrow["\sigma_{M, s}", from=2-1, to=2-2]
	\arrow[equal, from=1-1, to=1-2]
	\arrow["{\sigma_{\sfV_\et, s}}", from=1-2, to=2-2]
    \end{tikzcd}
    \begin{tikzcd}
	{\w_\dR(M)} & {\sfV_{\dR, s}} \\
	{\w_\dR(M^\sigma)} & {\w_\dR(M_{\sigma(s)}) = \sfV_{\dR, \sigma(s)}}
	\arrow["{\iso_{\mathrm{bc}}}"', from=1-1, to=2-1]
	\arrow["\sigma_{M, s}", from=2-1, to=2-2]
	\arrow[equal, from=1-1, to=1-2]
	\arrow["{\sigma_{\sfV_\dR, s}}", from=1-2, to=2-2]
    \end{tikzcd}
    \end{equation} 
    In orther words, the Hodge realization of $\sigma_{M, s}$ provides the isomorphism $\gamma^\sigma$ we are seeking for the test object $(M = M_s, \gamma)$ at $s$. 
\end{proof}

\subsection{Level Structures on Systems}\label{sec: level str on systems} Finally, we make some definitions regarding level structures on systems of realizations. The first two are not used until section~\ref{sec: Shimura}. Let $(G, V, \sfK)$ be as in \ref{sec: level structure bc}, and we recall the notations introduced there. Let $S$ be a smooth variety over $k \subseteq  \IC$ and $\sfV = (\sfV_B, \sfV_\dR, \sfV_\et) \in \sfR(S)$. Let $[\eta]$ be a $\sfK$-level structure on $\sfV_\et$. 

\begin{definition}
\label{def: level on sfV}
    We say that $[\eta]$ on $\sfV_\et$ is \textbf{$\sfV$-rational} if it satisfies the following property: For every $v \in I := (V^\tensor)^G$, the global section $\eta(v)_\et \in \H^0(\sfV_\et^\tensor)$ (see \ref{sec: define eta(v)}) is the \'etale component of a global section $\bv = (\bv_B, \bv_\dR, \bv_\et) \in \H^0(\sfV^\tensor)$ (i.e., $\eta(v)_\et = \bv_\et$). We denote $\bv$ by $\eta(v)$. Moreover, for every $s \in S(\IC)$, $(\sfV_{B, s}, \{ \eta(v)_{B, s}\}_{v \in I}) \iso (V, \{ v\}_{v \in I})$. 
\end{definition}

\begin{definition}
    \label{def: HS type} If $V$ is a faithful representation, $\Ohm$ is a $G(\IR)$-conjugacy class of morphisms $\IS \to G_\IR$, and $[\eta]$ is $\sfV$-rational, we say that $(\sfV, [\eta])$ is \textbf{of type $\Ohm$} if for every $s \in S(\IC)$, under some (and hence every) isomorphism $(\sfV_{B, s}, \{ \eta(v)_{B, s}\}_{v \in I}) \iso (V, \{ v\}_{v \in I})$ the Hodge structure on $\sfV_{B, s}$ is defined by an element of $\Ohm$. 
\end{definition}

\begin{definition}
\label{def: rigid}
    A $\sfK$-level structure $[\eta]$ on $\sfV_\et$ is said to be \textbf{rigid} on $\sfV$ if for any $s \in S(\IC)$, there is no nontrivial automorphism of $\sfV_{B, s}$ which preserves the Hodge filtration on $\sfV_{\dR, s}$ and the level structure $[\eta]_s$. 
\end{definition}

\subsection{Galois Descent Lemmas}

\begin{lemma}
\label{lem: descent of morphisms of systems} 
    Let $S$ be a smooth variety over a field $k \subseteq \IC$ and take $\sfV, \sfW \in \sfR_\am^*(S)$. Let $\varphi_\IC$ be a morphism $\sfV|_{S_\IC} \to \sfW|_{S_\IC}$. Then $\varphi_\IC$ descends to a morphism $\sfV \to \sfW$ if and only if either the \'etale or the de Rham component of $\varphi_\IC$ descends to $S$. 
\end{lemma}
\begin{proof}
    Considering $\varphi_\IC$ as a global section of $(\sfV^\vee \tensor \sfW)|_{S_\IC}$, we know that either its \'etale or de Rham component descent to $S$. By \ref{sec: descent of one component} this implies that $\varphi_\IC$ itself descends to $S$.
\end{proof}

\begin{corollary}
\label{cor: uniqueness of VB descent}
    Let $S$ be a smooth variety over a field $k \subseteq \IC$. Suppose that $(\sfV_B, \sfV_{\dR, \IC})$ is a VHS over $S_\IC$ equipped with a descent $\sfV_\et$ of $\sfV_B \tensor \IA_f$ to $S$. If there exists a descent $(\sfV_{\dR}, \varphi : \sfV_{\dR}|_{S_\IC} \iso \sfV_{\dR, \IC})$ of $\sfV_{\dR, \IC}$ as a filtered flat vector bundle over $S$ such that $\sfV := (\sfV_B, \sfV_\dR, \sfV_\et) \in \sfR(S)$ belongs to the subcategory $\sfR^*_\am(S)$, then the descent $(\sfV_\dR, \varphi)$ is unique up to unique isomorphism. 
\end{corollary}
Recall that an automorphism of $(\sfV_\dR, \varphi)$ is an automorphism of $\sfV_\dR$ which fixes $\varphi$, so it has to be the identity. Therefore, it suffices to prove uniqueness up to isomorphism. 
\begin{proof}
    Assume that we have $\sfV = (\sfV_B,\sfV_\dR,\sfV_\et),\sfV' = (\sfV_B,\sfV'_\dR,\sfV_\et) \in \sfR^\ast_{\am}(S)$ such that there exist isomorphisms $\sfV_\dR|_{S_\IC} \stackrel{\varphi}{\cong} \sfV_{\dR,\IC} \stackrel{\varphi'}{\cong} \sfV'_{\dR}|_{S_\IC}$. By Lemma~\ref{lem: descent of morphisms of systems} the isomorphism $(\id_{\sfV_B}, \varphi' \circ \varphi,\id_{\sfV_{\et}}|_{S_\IC})\colon  \sfV|_{S_\IC} \to \sfV'|_{S_\IC}$ descends to $S$ since its \'etale component does, proving the claim.\end{proof}


\begin{lemma}
\label{prop: engine of desceding VdR}
    Let $S$ be a smooth variety over a field $k \subseteq \IC$ and $\sfV = (\sfV_B, \sfV_\dR, \sfV_\et)$ be an object of $\sfR^*_\am(S)$. Suppose that $S$ is an \'etale cover of another smooth $k$-variety $T$, and the restriction $\sfV|_{S_\IC}$ descends to $\sfW_\IC = (\sfW_B, \sfW_{\dR, \IC}, \sfW_{\et, \IC}) \in \sfR(T_\IC)$. Consider the diagram 
    \[\begin{tikzcd}
	{S_\IC} & S \\
	{T_\IC} & T
	\arrow[from=1-1, to=2-1]
	\arrow[from=2-1, to=2-2]
	\arrow[from=1-2, to=2-2]
	\arrow[from=1-1, to=1-2]
	\arrow["\lrcorner"{anchor=center, pos=0.125}, draw=none, from=1-1, to=2-2]
    \end{tikzcd}\]
    If $(\sfV_\et, \sfW_{\et, \IC})$ admits a simultaneous descent to a local system $\sfW_\et$ over $T$, then $(\sfV_\dR, \sfW_{\dR, \IC})$ also admits a simultaneous descent to some filtered flat vector bundle $\sfW_\dR$ over $T$, such that $(\sfW_B, \sfW_\dR, \sfW_\et) \in \sfR^*_\am(T)$. Moreover, if $\sfV \in \sfR_\am(S)$, then $\sfW \in \sfR_\am(T)$. 
\end{lemma}
We remark that the identifications $\sfW_B \tensor \sO^\an_{T_\IC} = \sfW_\dR|_{T_\IC}$ and $\sfW_B \tensor \IA_f = \sfW_\et |_{T_\IC}$ are encoded in the definition of simultaneous descent (see \ref{sec: simultaneous descent}). 


\begin{proof}
     Let $p_1, p_2$ be the two projections of $U := S \times_T S$ to $S$. The descent of $\sfV_\IC := \sfV|_{S_\IC}$ to $\sfW_\IC$ is provided by an isomorphism $\varphi_\IC : (p_1^* \sfV)|_{U_\IC} \sto (p_2^* \sfV)|_{U_\IC}$ that satisfies the cocycle condition. That $(\sfV_\et, \sfW_{\et, \IC})$ admits a simultanenous descent to a local system $\sfW_\et$ over $T$ implies that the \'etale component of $\varphi_\IC$ descends to $U$, so that \ref{lem: descent of morphisms of systems} implies that $\varphi_\IC$ descends to an isomorphism $\varphi : p_1^* \sfV \sto p_2^* \sfV$ over $U$. Note that $\varphi$ still satisfies the cocycle condition, because it is a condition that can be checked after we base change to $U_\IC$. The de Rham component of $\varphi$, which respects the connections and filtrations, provides the descent datum to construct the desired descent of filtered flat vector bundle $\sfW_\dR$, so that $\sfW :=  (\sfW_B, \sfW_\dR, \sfW_\et)$ is a system of realizations over $T$. By \ref{sec: motivation for weakly AM}(ii), that $\sfV \in \sfR^*_\am(S)$ implies $\sfW \in \sfR^*_\am(T)$. If moreover $\sfV \in \sfR_\am(S)$, then $\sfW \in \sfR_\am(T)$ by the very definition of $\sfR_\am(-)$. 
\end{proof}

\begin{lemma}
\emph{(Rigidity Lemma)}
\label{lem: rigidity}
    Let $S$ be a smooth variety over a field $k \subseteq \IC$ and $(G, V, \sfK)$ be as in \ref{sec: level structure bc}. Let $\sfV = (\sfV_B, \sfV_\dR, \sfV_\et)$ and $\sfW = (\sfW_B, \sfW_\dR, \sfW_\et)$ be objects of $\sfR^*_\am(S)$ equipped with rigid $\sfK$-level structures $[\eta]$ and $[\xi]$ respectively. Then every isomorphism $\lambda : (\sfV, [\eta])|_{S_\IC} \sto (\sfW, [\xi])|_{S_\IC}$ over $S_\IC$ descends to an isomorphism $(\sfV, [\eta])\sto (\sfW, [\xi])$ over $S$. 
\end{lemma}
\begin{proof}
    The conclusion amounts to saying that the isomorphisms $\sfV_{\dR}|_{S_\IC} \sto \sfW_{\dR}|_{S_\IC}$ and $\sfV_{\et}|_{S_\IC} \sto \sfW_{\et}|_{S_\IC}$ over $S_\IC$ given by $\lambda$ descend to isomorphisms $\sfV_\dR \sto \sfW_\dR$ and $\sfV_\et \sto \sfW_\et$ over $S$ respectively, and the latter moreover sends $[\eta]$ to $[\xi]$. In fact, it suffices to show that $\sfV_{\et}|_{S_\IC} \sto \sfW_{\et}|_{S_\IC}$ descends. Indeed, it is then automatic that $[\eta]$ is sent to $[\xi]$ (cf. \ref{sec: define eta(v)}), and one applies \ref{lem: descent of morphisms of systems} to conclude that $\sfV_{\dR}|_{S_\IC} \sto \sfW_{\dR}|_{S_\IC}$ descends as well.

    We are given that $\lambda$ sends $[\eta^\an] :=[\eta]_{S_\IC}$ to $[\xi^\an] :=[\xi]_{S_\IC}$. Write $\lambda_x$ for the fiber of $\lambda$ at a point $x \in S(\IC)$ and $\lambda_{B, x}$ for its Betti component. Note that $\lambda_{B, x} \tensor \IA_f$ sends $[\eta^\an]_x$ to $[\xi^\an]_x$.  
    
    Fix an arbitrary $s \in S(\IC)$ and $\sigma \in \Aut(\IC/k)$. Let $(M, \gamma)$ be a test object at $s$ for $\sfV$ to be weakly AM and remember the resulting $\gamma^\sigma : \w_\Hdg(M^\sigma) \sto (\sfV_{B, \sigma(s)}, \sfV_{\dR, \sigma(s)})$ (see \ref{def: compatible with AM}). Set $\nu = \lambda_s \circ \gamma$. Then $(M, \nu)$ becomes a test object at $s$ for $\sfW$, and take the resulting $\nu^\sigma : \w_\Hdg(M^\sigma) \sto (\sfW_{B, \sigma(s)}, \sfW_{\dR, \sigma(s)})$ such that the diagrams in \ref{def: compatible with AM} commute with $\gamma^\sigma$ replaced by $\nu^\sigma$. Let $[\eta_M]$ be the level structure on $\w_\et(M)$ given by pulling back $[\eta^\an]_s$ along $\gamma_B \tensor \IA_f$. Then since $\lambda_{B, s} \tensor \IA_f$ sends $[\eta^\an]_s$ to $[\xi^\an]_s$, $[\eta_M]$ is also the pullback of $[\xi^\an]_s$ along $\nu_B \tensor \IA_f$. By \ref{sec: descent of one component}, $\gamma^\sigma_B \tensor \IA_f$ sends $[\eta_M]^\sigma$ to $[\eta^\an]_{\sigma(s)}$ and $\nu^\sigma_B \tensor \IA_f$ sends $[\eta_M]^\sigma$ to $[\xi^\an]_{\sigma(s)}$.

    Here comes the key observation: Since $\lambda_{B, \sigma(s)} \tensor \IA_f$ sends $[\eta^\an]_{\sigma(s)}$ to $[\xi^\an]_{\sigma(s)}$, $\lambda_{\sigma(s)} \circ \gamma^\sigma$ is another isomorphism $\w_\Hdg(M^\sigma) \sto (\sfW_{B, \sigma(s)}, \sfW_{\dR, \sigma(s)})$ whose \'etale realization sends $[\eta_M]^\sigma$ to $[\xi^\an]_{\sigma(s)}$. By the rigidity assumption, there is no nontrivial automorphism of the Hodge structure $(\sfW_{B, \sigma(s)}, \sfW_{\dR, \sigma(s)})$ whose \'etale realization fixes $[\xi^\an]_{\sigma(s)}$. Therefore, we must have $\nu^\sigma = \lambda_{\sigma(s)} \circ \gamma^\sigma$. 

    Now one readily checks that the following diagram commutes:  
    \[\begin{tikzcd}
	& {\w_\et(M) } \\
	{\sfV_{B, s} \tensor \IA_f} && {\sfW_{B, s} \tensor \IA_f} \\
	{} & {\w_\et(M^\sigma)} \\
	{\sfV_{B, \sigma(s)} \tensor \IA_f} && {\sfW_{B, \sigma(s)} \tensor \IA_f}
	\arrow["{\lambda_{B, s} \tensor \IA_f }", from=2-1, to=2-3]
	\arrow["{\gamma_B \tensor \IA_f}"', from=1-2, to=2-1]
	\arrow["{\nu_B \tensor \IA_f}", from=1-2, to=2-3]
	\arrow["{\sigma_{\sfV_\et, s}}"', from=2-1, to=4-1]
	\arrow["{\lambda_{B, \sigma(s)} \tensor \IA_f}", from=4-1, to=4-3]
	\arrow["{\sigma_{\sfW_\et, s}}", from=2-3, to=4-3]
	\arrow["{\gamma^\sigma_B \tensor \IA_f}"', from=3-2, to=4-1]
	\arrow["{\nu_B^\sigma \tensor \IA_f}", from=3-2, to=4-3]
	\arrow["{\iso_\bc}"'{pos=0.7}, from=1-2, to=3-2]
    \end{tikzcd}\]
    Indeed, by the construction of $\gamma^\sigma$ and $\nu^\sigma$, the two squares which contain $\w_\et(M) \iso_\bc \w_\et(M^\sigma)$ commute. The upper triangle commutes by definition of $\nu$, and the lower triangle commutes by the preceeding paragraph. Therefore, the square at the front has to commute as well. As $s \in S(\IC)$ and $\sigma \in \Aut(\IC/k)$ were arbitrary, this implies that $\lambda_B \tensor \IA_f$ descends to an isomorphism $\sfV_\et \sto \sfW_\et$ as desired. 
\end{proof}

\begin{remark}
The above rigidity lemma roughly says that Galois descent of isomorphisms between weakly AM systems of realizations is forced upon us by the existence of certain rigid level structures. The following well known example can be seen as a special case of the above lemma, in which the role played by the rigid level structure is made more explicit: Suppose for $i = 1, 2$ and $N \ge 3$, $(E_i, \alpha_i)$ is an elliptic curve over a smooth variety $S$ over $k \subseteq \IC$ with a (full) $N$-level structure $\alpha_i : \underline{(\IZ/N\IZ)}_S^{\oplus 2} \sto (E_i/S)[N]$. Set $\sfK := \ker(\GL_2(\what{\IZ}) \to \GL_2(\IZ/N \IZ))$. Then the first cohomology of $E_i$ defines an object $\sfH_i \in \sfR_\am^\circ(S)$ and $\alpha_i$ defines a $\sfK$-level structure on $\sfH_i$. 

If $(E_1, \alpha_1)|_{S_\IC} \iso (E_2, \alpha_2)|_{S_\IC}$ (or equivalently $(\sfH_1, \alpha_1)|_{S_\IC} \iso (\sfH_2, \alpha_2)|_{S_\IC}$), then in fact $(E_1, \alpha_1) \iso (E_2, \alpha_2)$ over $S$ already (and hence $(\sfH_1, \alpha_1) \iso (\sfH_2, \alpha_2)$). This follows a fortiori from the existence of a fine moduli space of elliptic curves with level $N$-structure over $k$. Indeed, two $S$-valued points on this moduli space must be equal if they are equal after tensoring with $\IC$. But without using the moduli theory this can be seen as follows: Without loss of generality, suppose that $S$ is connected and $K$ is the fraction field of $S$. Then $K$-forms of $(E_i, \alpha_i)|_K$ are classified by $\H^1(\Gal_K, \Aut((E_i, \alpha_i)|_{\bar{K}}))$, which is trivial as $\Aut((E_i, \alpha_i)|_{\bar{K}})$ is trivial. The resulting isomorphism $(E_1, \alpha_1)|_K \iso (E_2, \alpha_2)|_K$ (which has to be unique) extends over $S$.
\end{remark}

\section{Shimura Varieties}
\label{sec: Shimura}
Let $(G, \Ohm)$ be a Shimura datum which satisfies the axioms in \cite[Ch II, (2.1)]{Milne:CanonicalModels}. Let $E(G, \Ohm)$ be the reflex field, $Z$ be the center of $G$ and $Z_s$ be the maximal anisotropic subtorus of $Z$ that is split over $\IR$. In this note, we always assume that 
\begin{equation}
\label{eqn: assumption on center}
    \text{the weight is defined over $\IQ$ and $Z_s$ is trivial}. 
\end{equation}
Note that in particular the latter condition ensures that $Z(\IQ)$ is discrete in $Z(\IA_f)$ (\cite[Rmk~5.27]{MilIntro}). We will often drop the Hermitian symmetric domain $\Ohm$ from the notation of Shimura varieties when no confusion would arise. 

For any neat compact open subgroup $\sfK \subseteq G(\IA_f)$, let $\Sh_\sfK(G)_\IC$ denote the resulting Shimura variety with a complex uniformization $G(\IQ) \backslash \Ohm \times G(\IA_f) / \sfK$ and let $\Sh_\sfK(G)$ denote the canonical model over $E(G, \Ohm)$. Let $\Sh(G)$ denote the inverse limit $\varprojlim_\sfK \Sh_\sfK(G)$ as $\sfK$ runs through all compact open subgroups. Under our assumptions, $\Sh(G)(\IC)$ is described by $G(\IQ) \backslash \Ohm \times G(\IA_f)$ (\cite[(5.28)]{MilIntro}). Note that $\Sh_\sfK(G) = \Sh(G)/ \sfK$.  


    


\subsection{Automorphic VHS and \'etale local systems} Let $G \to \GL(V)$ be a representation. For any neat compact open subgroup $\sfK \subseteq G(\IA_f)$, we can attach to $\Sh_\sfK(G)_\IC$ a polarizbale VHS $\sfV_\Hdg = (\sfV_B, \sfV_{\dR, \IC})$ (cf. \cite[Ch. II~3.3]{Milne:CanonicalModels}, \cite[\S2.2]{Taelman2}). In particular, $\sfV_B$ is defined to be the contraction product $V \times^{G(\IQ)} [\Ohm \times G(\IA_f) / \sfK]$. The \'etale local system defined by $\sfV_B \tensor \IA_f$ can be alternatively constructed as $(V \tensor \IA_f) \times^\sfK [G(\IQ) \backslash  \Omega \times G(\IA_f) ]$. This allows us to descend it to $\Sh_\sfK(G)$. To be pedantic, choose an $\what{\IZ}$-lattice $\Lambda$ which is stabilized by $\sfK$ and let $\rho: \sfK \to \GL(\Lambda)$ be the representation. Set $\sfK_n = \sfK \cap \rho^{-1} \ker(\GL(\Lambda) \to \GL(\Lambda/n))$. The natural map $\Sh_{\sfK_n}(G)_\IC \to \Sh_\sfK(G)_\IC$ is a finite \'etale map with deck transformation group $\sfK/\sfK_n$. Then we define a $\Lambda / n$-local system on $\Sh_\sfK(G)_\IC$ by 
$$ L_n := (\Lambda / n) \times^{\sfK/\sfK_n} \Sh_\sfK(G)_\IC.  $$
Then $\sfV_B \tensor \IA_f = (\varprojlim_n L_n) \tensor \IQ$, where $n$ goes to $\infty$ multiplicatively. Since $\Sh_{\sfK_n}(G)_\IC \to \Sh_\sfK(G)_\IC$ is defined over $E(G, \Ohm)$, the \'etale sheaf $\sfV_B \tensor \IA_f$ descends to $\Sh_\sfK(G)$. The resulting etale sheaf $\sfV_\et$ over $\Sh_\sfK(G)$ is well defined and is independent of the choice of $\Lambda$, so we obtain a well defined tensor functor from the category $\mathrm{Rep}(G)$ of $\IQ$-representations of $G$ to $\Loc(\Sh_\sfK(G))$ (cf. the description in \cite[340-341]{LiuZhu}). We say that $(\sfV_B, \sfV_{\dR, \IC})$ is the \textbf{automorphic VHS} on $\Sh_\sfK(G)_\IC$ and $\sfV_\et$ is the \textbf{automorphic \'etale local system} on $\Sh_\sfK(G)$ associated to $V$.

\textit{We use a superscript $\sfK$ (e.g., $\sfV_\Hdg^\sfK$) when we need to emphasize that we are considering the objects at level $\sfK$.} We also remark that $\sfV^\sfK_\et$ admits a \textbf{tautological $\sfK$-level structure} $[\eta^\sfK_V]$, since by construction its restriction to $\Sh(G)$ comes equipped with a trivialization $V \tensor \underline{\IA_f} \sto \sfV^\sfK_\et|_{\Sh(G)}$, and $\sfV^\sfK_\et|_{\Sh(G)}$ admits a $G(\IA_f)$-equivariant structure. Conversely, the trivialization over $\Sh(G)$ is recovered by $\varprojlim_{\sfK' \subseteq \sfK} [\eta_V^{\sfK'}]$. 


The reader may easily check that the above constructions are suitably functorial in $V$: 
\begin{theorem}
\label{thm: functoriality of automorphic sheaves}
    Let the objects $\sfV_B, \sfV_{\dR, \IC}, \sfV_\et$ and $[\eta_V]$ be as defined above for a neat $\sfK$. Let $W$ be another representation of $G$ and construct the corresponding $\sfW_B, \sfW_{\dR, \IC}, \sfW_\et$, and $[\eta_W]$. Every morphism $V \to W$ induces a morphism $(\sfV_B, \sfV_{\dR, \IC}) \to (\sfW_B, \sfW_{\dR, \IC})$ of VHS over $\Sh_\sfK(G)_\IC$ and a morphism $\sfV_\et \to \sfW_\et$ of \'etale local systems over $\Sh_\sfK(G)$. These induced morphisms are uniquely characterized by the following properties:
    \begin{enumerate}[label=\upshape{(\alph*)}]
        \item The restriction of $\sfV_\et \to \sfW_\et$ to $\Sh_\sfK(G)_\IC$ is equal to $\sfV_B \tensor \IA_f \to \sfW_B \tensor \IA_f$ under the canonical identifications. 
        \item The diagram 
       \[\begin{tikzcd}
	{V \tensor \underline{\IA_f}} & {\sfV_\et|_{\Sh(G)}} \\
	{W \tensor \underline{\IA_f}} & {\sfW_\et|_{\Sh(G)}}
	\arrow["\sim", from=1-1, to=1-2]
	\arrow[from=1-2, to=2-2]
	\arrow[from=1-1, to=2-1]
	\arrow["\sim", from=2-1, to=2-2]
    \end{tikzcd}\]
    commutes, where the horizontal arrows are the trivializations provided by $\varprojlim_{\sfK' \subseteq \sfK} [\eta^{\sfK'}_V]$ and $\varprojlim_{\sfK' \subseteq \sfK} [\eta^{\sfK'}_W]$ respectively. 
    \end{enumerate}
\end{theorem}


Define $[\eta^\an_V] := [\eta_V]_{\Sh_\sfK(G)_\IC}$. Using the automorphic VHS $(\sfV_B, \sfV_{\dR, \IC})$, we can already give a moduli interpretation of $\Sh_\sfK(G)_\IC$: 
\begin{theorem}
\label{thm: moduli over C}
    Assume that $V$ is faithful. For every smooth $\IC$-variety $T$, let $\sM_V(T)$ be the groupoid of pairs $(\IW, [\xi^\an])$ where $\IW = (\IW_B, \IW_\dR)$ is a VHS over $T$ and $[\xi^\an]$ is a $\sfK$-level structure on $\IW_B \tensor \IA_f$ which is $\IW$-rational, and $(\IW, [\xi^\an])$ is of type $\Ohm$. 
    
    Then $(\sfV_\IC := (\sfV_B, \sfV_{\dR, \IC}),  [\eta^\an_V])$ is an object of $\sM_V(\Sh_\sfK(G)_\IC)$ and for every object $(\IW, [\xi^\an]) \in \sM_V(T)$ there exists a unique morphism $\rho : T \to \Sh_\sfK(G)_\IC$ such that $\rho^* (\sfV_\IC, [\eta^\an_V]) \iso (\IW, [\xi^\an])$. 
\end{theorem} 
\begin{proof}
    This follows from \cite[Prop.~3.10]{Mil94}. Let $I$ be an indexing set of $(V^\tensor)^G$ and write the elements in $(V^\otimes)^G$ as $\{ s_\alpha \}_{\alpha \in I}$. That $[\xi^\an]$ is $\IW$-rational implies that for every $\alpha$, $\xi^\an(s_\alpha)_\et$ comes from a global section in $\H^0(\IW_B^\tensor)$ which is everywhere of Hodge type $(0, 0)$ (see \ref{def: level on sfV}) and let us denote it by $\bt_{\alpha, B}$. It is not hard to check that $[\eta^\an_V]$ is $\sfV_\IC$-rational, and we let $\bs_{\alpha, B} \in \H^0(\sfV_B)$ be the global section defined by $\eta_V(s_\alpha)_\et$. 
    
    Here is a sketch of the argument in \textit{loc. cit.} in our notations: Using the complex uniformization of $\Sh_\sfK(G)(\IC)$ and the fact that both $[\eta^\an_V]$ and $[\xi^\an]$ are of $\Ohm$-type, one checks that for each $t \in T(\IC)$, there is a unique $s \in \Sh_\sfK(G)(\IC)$ such that $(\IW_t, \{ \bt_{\alpha, B, t} \}, [\xi^\an]_t) \iso ((\sfV_{B, s}, \sfV_{\dR, s}), \{ \bs_{\alpha, B, s} \}, [\eta^\an_V]_s)$, so that $t \mapsto s$ defines a map of sets $T(\IC) \to \Sh_\sfK(G)(\IC)$. Then one uses the assumption that $\IW$ is a VHS (so that the Hodge filtration on $\IW_t$ varies holomorphically with $t$) to show that the map is holomorphic. Finally, a theorem of Baily and Borel (recalled in Thm~2.24 in \textit{loc. cit.}) says that any such holomorphic map is algebraic. 
\end{proof}

\begin{remark}
\label{rmk: rigid}
    We remark that the assumption (\ref{eqn: assumption on center}) together with the neatness of $\sfK$ implies that for every $(\IW, [\xi^\an]) \in \sM_V(T)$ as above, the $\sfK$-level structure $[\xi^\an]$ is rigid, as defined in \ref{def: rigid} (cf. \cite[\S2.1]{Taelman2}, \cite[Rmk~3.11]{Mil94}). In particular, the groupoid $\sM_V(T)$ is in fact a set. The above theorem implies that $\Sh_\sfK(G)_\IC$ represents the functor $\sM_V(-)$ from the category of smooth $\IC$-varieties to sets.
\end{remark}

\begin{remark}
\label{rmk: rational level structure}
    Unlike the set up in \cite[Prop.~3.10]{Mil94}, we do not keep a collection of global Hodge tensors ($\mathfrak{s}$ in Milne's notation, which plays the role of $\{ \bt_{\alpha, B} \}$ in ours) of $\IW$ as part of the data for objects parametrized by $\Sh_\sfK(G)_\IC$. Instead, we impose the condition that the $\sfK$-level structure $[\xi^\an]$ needs to be $\IW$-rational. Doing so does not change the mathematical content because $\bt_{\alpha, B}$'s are uniquely determined by the level structure $[\xi^\an]$ provided that they exist. 
\end{remark}

\subsection{Constructing Systems of Realizations}

We shall work under the following assumption: 
\begin{definition}
    \label{def: nice cover}
    We say that a Shimura datum $(G, \Ohm)$ of abelian type which satisfies (\ref{eqn: assumption on center}) has a nice Hodge-type cover if there exists a morphism of Shimura data $(\wt{G}, \wt{\Ohm}) \to (G, \Ohm)$ such that 
    \begin{enumerate}[label=\upshape{(\roman*)}]
        \item $(\wt{G}, \wt{\Ohm})$ is of Hodge type and also satisfies assumption (\ref{eqn: assumption on center});
        \item $\wt{G} \to G$ is surjective and the kernel lies in the center of $\wt{G}$; 
        \item the induced embedding on reflex fields $E(G, \Ohm) \subseteq E(\wt{G}, \wt{\Ohm})$ is an equality.
    \end{enumerate}
\end{definition}
In the above, by ``cover'' we just mean that $\wt{G}$ can be viewed as a cover of $G$. The map between Hermitian symmetric domains $\wt{\Ohm} \to \Ohm$ is usually not surjective, but only a cover over some connected components of $\Ohm$. In practice, $\Sh(\wt{G}, \wt{\Ohm})$ does function as an ``\'etale cover'' of $\Sh(G, \Ohm)$ up to taking Hecke translates (e.g., in the proof of \ref{thm: autVB exists} below).

In general, a Shimura datum $(G, \Ohm)$ is said to be of abelian type if there exists a Shimura datum $(\wt{G}, \wt{\Ohm})$ of Hodge type and a central isogeny between derived subgroups $\wt{G}^{\mathrm{der}} \to G^{\mathrm{der}}$ which induces an isomorphism between the adjoint Shimura data $(\wt{G}^{\mathrm{ad}}, \wt{\Ohm}^{\mathrm{ad}}) \sto (G^{\mathrm{ad}}, \Ohm^{\mathrm{ad}})$, so that the relationship between $(G, \Ohm)$ and $(\wt{G}, \wt{\Ohm})$ is more indirect than the one in \ref{def: nice cover}; in particular, $E(G, \Ohm)$ could be strictly smaller than $E(\wt{G}, \wt{\Ohm})$. 

\begin{theorem}
\label{thm: autVB exists}
    Assume that $(G, \Ohm)$ is a Shimura datum of abelian type which satisfies (\ref{eqn: assumption on center}) and has a nice Hodge-type cover. Let $\sfK \subseteq G(\IA_f)$ be a neat compact open subgroup. Let $G \to \GL(V)$ be a representation and $(\sfV_B, \sfV_{\dR, \IC})$ (resp. $\sfV_\et$) be the associated automorphic VHS over $\Sh_\sfK(G)_\IC$ (resp. automorphic \'etale local system on $\Sh_\sfK(G)$). 

    Then there exists a unique descent $\sfV_\dR$ of $\sfV_{\dR, \IC}$ to $\Sh_\sfK(G)$ such that $\sfV := (\sfV_B, \sfV_\dR, \sfV_\et)$ is an object of $\sfR^*_\am(\Sh_\sfK(G))$. Moreover, $\sfV$ lies in the subcategory $\sfR_\am(\Sh_\sfK(G))$. 
\end{theorem}

We call $\sfV := (\sfV_B, \sfV_\dR, \sfV_\et)$ as above the automorphic system of realizations associated to $V$. Note that by \ref{cor: uniqueness of VB descent}, the descent $\sfV_\dR$ of $\sfV_{\dR, \IC}$ that makes $\sfV$ an object of $\sfR^*_\am(\Sh_\sfK(G))$ is automatically unique provided that it exists. We write $\sfV$ as $\sfV^\sfK = (\sfV^\sfK_B, \sfV^\sfK_\dR, \sfV^\sfK_\et)$ when we need to emphasize that we are considering objects at level $\sfK$. We start with a preliminary reduction. 

\begin{lemma}
\label{lem: uniquess of autVB} The truth of to \ref{thm: autVB exists} is independent of $\sfK$ (as long as $\sfK$ is neat). 
\end{lemma}
\begin{proof}
    Suppose that $\sfK' \subseteq \sfK$. If we have already constructed the desired $\sfV_\dR^\sfK$ for $\sfK$ over $S := \Sh_{\sfK}(G)$, then we can simply pullback $\sfV_\dR^\sfK$ to $S' := \Sh_{\sfK'}(G)$, since the pullback of $(\sfV_B^\sfK, \sfV^\sfK_{\dR, \IC})$ to $S'_\IC$ and that of $\sfV^\sfK_\et$ to $S'$ are identified with $(\sfV_B^{\sfK'}, \sfV_{\dR, \IC}^{\sfK'})$ and $\sfV_\et^{\sfK'}$ respectively. It is clear that the pullback of an object in $\sfR_\am(S)$ is an object in $\sfR_\am(S')$. Conversely, suppose that we have constructed $\sfV_\dR^{\sfK'}$ over $S'$ so that $(\sfV^{\sfK'}_B, \sfV^{\sfK'}_\dR, \sfV^{\sfK'}_\et) \in \sfR_\am(\Sh_{\sfK'}(G))$. Note that $(\sfV^{\sfK'}_B, \sfV^{\sfK'}_{\dR, \IC})$ descends to $S'_\IC$ and $\sfV^{\sfK'}_\et$ descends to $S$ because $(\sfV^\sfK_B, \sfV^\sfK_{\dR, \IC})$ and $\sfV^\sfK_\et$ already exist. We conclude by \ref{prop: engine of desceding VdR} that the object $(\sfV^{\sfK'}_B, \sfV^{\sfK'}_\dR, \sfV^{\sfK'}_\et) \in \sfR_\am(S')$ descends to $S$. In particular, this gives the desired $\sfV_\dR^\sfK$. 
\end{proof}

\begin{lemma}
\label{thm: autVH Hodge type}
    Theorem~\ref{thm: autVB exists} holds if the Shimura datum $(G, \Ohm)$ in question is of Hodge type.  
\end{lemma}
\begin{proof}
    By \ref{lem: uniquess of autVB} it suffices to consider the case when $\sfK$ is sufficiently small. By assumption, there exists an embedding of Shimura data $(G, \Ohm) \into (\GSp, \sH^\pm)$ where $\GSp = \GSp(H)$ for some symplectic space $H$ and $\sH^\pm$ is the Siegel half planes. Let $K \subseteq \GSp(\IA_f)$ be some sufficiently small compact open subgroup and assume $\sfK \subseteq K$ up to shrinking $\sfK$, so that we obtain a morphism $\Sh_\sfK(G) \to \Sh_K(\GSp)_E$, where $E := E(G, \Ohm)$ is the reflex field. Let $\sA$ be the universal abelian scheme over $\Sh_K(\GSp)$. Let $\sfH = (\sfH_B, \sfH_\dR, \sfH_\et)$ be the system of realizations over $S = \Sh_\sfK(G)$ given by the first cohomology of $\sA_S$.
    
    Using that $H$ is a faithful representation of $G$, we may find a $N \gg 0$ with the property that $V$ is isomorphic to a subquotient of $T_N(H) := \oplus_{a,b \le N} H^{\tensor a} \tensor (H^\vee)^{\tensor b}$ and define $T_N(\sfH)$ similarly. Since $G$ is reductive, there exists a $G$-invariant idempotent projector $\pi \in \End(T_N(H))$ such that $V \iso \mathrm{im}(\pi)$ as $G$-representations. Fix such an isomorphism. Note that we may view $\pi$ as a morphism $\IQ \to \End(T_N(H))$ (here $\IQ$ stands for the trivial representation of $G$). Therefore, by \ref{thm: functoriality of automorphic sheaves} we obtain a $\bpi_B \in \H^0(\End(T_N(\sfH_B)))$ which is everywhere of Hodge type $(0, 0)$ (i.e., $\bpi_B$ defines an endomorphism of $T_N(\sfH|_{S_\IC})$ as a VHS) and whose \'etale realization descends to a global section $\bpi_\et \in H^0(\End(T_N(\sfH_\et)))$. By \ref{sec: descent of one component}, the de Rham realization of $\bpi_B$ also descends to some $\bpi_\dR \in \H^0(\End(T_N(\sfH_\dR)))$, so that $\bpi := (\bpi_B, \bpi_\dR, \bpi_\et) \in \H^0(\End(T_N(\sfH)))$. Note that by \ref{thm: functoriality of automorphic sheaves} the fixed identification $V = \mathrm{im}(\pi)$ gives rise to identifications $\sfV |_{S_\IC} = \mathrm{im}(\bpi)|_{S_\IC}$ and $\sfV_\et = \mathrm{im}(\bpi_\et)$. Therefore, we may construct $\sfV_\dR$ by taking $\mathrm{im}(\bpi_\dR)$. Then by construction $\sfV = (\sfV_B, \sfV_\dR, \sfV_\et)$ is an object of $\sfR_\am(S)$. 
\end{proof}

Now we prove \ref{thm: autVB exists} in full: 
\begin{proof}
    Again by \ref{lem: uniquess of autVB} we may assume that $\sfK$ is sufficiently small. Then we choose a sufficiently small $\wt{\sfK} \subseteq \wt{G}(\IA_f)$ whose image lies in $\sfK$, so that we obtain a morphism $p_{\wt{\sfK}, \sfK} : \Sh_{\wt{\sfK}}(\wt{G}) \to \Sh_\sfK(G)$ from the morphism of Shimura data. Below we write $\wt{S}$ for $\Sh_{\wt{\sfK}}(\wt{G})$ and $S$ for $\Sh_\sfK(G)$. We note that (ii) implies that every connected component of $\wt{\Ohm}$ is mapped isomorphically onto some connected component of $\Ohm$, even though $\wt{\Ohm} \to \Ohm$ may not be surjective. Therefore, $p_{\wt{\sfK}, \sfK}$ is a finite \'etale cover of its image $\mathrm{im}(p_{\wt{\sfK}, \sfK})$, which is a union of connected components of $\Sh_\sfK(G)$.

    Viewing $V$ as a $\wt{G}$-representation via the projection $\wt{G} \to G$, we obtain the associated automorphic VHS on $\wt{S}_\IC$ and automorphic \'etale local system on $\wt{S}$. Denote them by $(\wt{\sfV}_B, \wt{\sfV}_{\dR, \IC})$ and $\wt{\sfV}_\et$ respectively. It is not hard to see that these are naturally identified with the pullback of $(\sfV_B, \sfV_{\dR, \IC})$ and $\sfV_\et$ along $p_{\wt{\sfK}, \sfK}$. By \ref{thm: autVH Hodge type} there exists a descent $\wt{\sfV}_{\dR}$ to $\wt{S}$ of $\wt{\sfV}_{\dR, \IC}$ such that $\wt{\sfV} := (\wt{\sfV}_B, \wt{\sfV}_\dR, \wt{\sfV}_\et)$ is an object of $\sfR_\am(\wt{S})$. Then we may construct the desired descent $\sfV_\dR$ over $\mathrm{im}(p_{\wt{\sfK}, \sfK})$ by applying \ref{prop: engine of desceding VdR} again. 

    In order to construct $\sfV_\dR$ on connected components not contained in $\mathrm{im}(p_{\wt{\sfK}, \sfK})$, we consider Hecke translates. Let $g \in G(\IA_f)$ be any element and set $\sfK_g := \sfK \cap g \sfK g^{-1}$ and $S_g := \Sh_{\sfK_g}(G)$. Then $S_g$ is a finite \'etale cover over $S$, and we denote the natural projection to $S$ by $p_g$. Note that $S_{g, \IC}$ has complex uniformization $G(\IQ) \backslash \Ohm \times G(\IA_f) / \sfK_g$. The automorphism $(x, g') \mapsto (x, g'g)$ on $\Ohm \times G(\IA_f)$ descends to a morphism $S_{g, \IC} \to S_\IC$, which we denote by $T_g$. It is well known that $T_g$ is defined over the reflex field, i.e., descends to $S_g \to S$.

    Now we recall that for varying $\sfK'$, the automorphic VHS $\{ (\sfV^{\sfK'}_B, \sfV^{\sfK'}_{\dR, \IC}) \}_{\sfK'}$ and the \'etale local systems $\{ \sfV^{\sfK'}_\et \}_{\sfK'}$ on $\Sh(G)$ (or its $\IC$-fiber) are $G(\IA_f)$-equivariant. In particular, there are natural identifications $T_{g, \IC}^* (\sfV^{\sfK}_B, \sfV^{\sfK}_{\dR, \IC}) = (\sfV^{\sfK_g}_B, \sfV^{\sfK_g}_{\dR, \IC})$ and $T_g^* (\sfV^{\sfK}_\et) = \sfV^{\sfK_g}_\et$. Without loss of generality we may assume that the image of $\wt{\sfK} \subseteq \wt{G}(\IA_f)$ in $G(\IA_f)$ lies in $\sfK_g$. Then $p_{\wt{\sfK}, \sfK}$ factors through $p_{\wt{\sfK}, \sfK_g}$. Consider now the composition 
    \[\begin{tikzcd}
	{\wt{S}} & {S_g}  & S
	\arrow["{p_{\wt{\sfK}, \sfK_g}}", from=1-1, to=1-2]
	\arrow["{T_g}", from=1-2, to=1-3]
    \end{tikzcd}\]
    which we shall denote by $\fp_g$. Then we have identifications $\fp_{g, \IC}^* (\sfV_B, \sfV_{\dR, \IC}) = (\wt{\sfV}_B, \wt{\sfV}_{\dR, \IC})$ and $\fp_g^* \sfV_\et = \wt{\sfV}_\et$. We may now apply \ref{prop: engine of desceding VdR} again to descend $\wt{\sfV} \in \sfR_\am(\wt{S})$ to $\mathrm{im}(\fp_g)$, which in particular defines a descent of $\sfV_{\dR, \IC}$ when restricted to $\mathrm{im}(\fp_g)_\IC$. For each connected component $S^\circ$ of $S$, and each $g \in G(\IA_f)$ such that $S^\circ \subseteq \mathrm{im}(\fp_g)$, the resulting descent of $\sfV_{\dR, \IC}|_{S^\circ_\IC}$ does not depend on $g$, by \ref{cor: uniqueness of VB descent}. We may now conclude using the fact that there is a finite subset $J \subseteq G(\IA_f)$ such that $S = \cup_{g \in J} \mathrm{im}(\fp_g)$. 
\end{proof}

\begin{theorem}
\label{thm: autSystem exists}
    For each neat $\sfK$, the association $V \mapsto \sfV^\sfK $ given by \ref{thm: autVB exists} defines a functor $\sF_\sfK : \mathrm{Rep}(G) \to \sfR_\am(\Sh_\sfK(G))$. Moreover, if $\sfK' \subseteq \sfK$, $\sF_{\sfK'}$ is the composition of $\sF_\sfK$ and the pullback functor $\sfR_\am(\Sh_{\sfK'}(G)) \to \sfR_\am(\Sh_\sfK(G))$; in other words, $\varprojlim_\sfK \sF_\sfK$ defines a functor from $\mathrm{Rep}(G)$ to the category of $G(\IA_f)$-equivariant objects in $\sfR_\am(\Sh(G))$. 
\end{theorem}
\begin{proof}
    This follows readily from \ref{thm: functoriality of automorphic sheaves} and \ref{lem: descent of morphisms of systems}. 
\end{proof}
 
\subsection{Moduli Interpretations over Reflex Fields}

We are now ready to give the moduli interpretations of the Shimura varieties over the reflex fields, under our simplifying assumptions. 

\begin{theorem}
\label{thm: moduli over reflex field}
    Let $(G, \Ohm)$ be a Shimura datum of abelian type which satisfies (\ref{eqn: assumption on center}) and has a nice Hodge-type cover.
    Assume that $V \in \mathrm{Rep}(G)$ is faithful. Let $E' \subseteq \IC$ be a subfield which contains $E = E(G, \Ohm)$ and $T$ be a smooth $E'$-variety. Let $\sM_{V, E'}(T)$ be the groupoid of pairs $(\sfW, [\xi])$ where $\sfW = (\sfW_B, \sfW_\dR, \sfW_\et) \in \sfR^*_\am(T)$, and $[\xi]$ is a $\sfK$-level structure on $\sfW_\et$ such that $[\xi]$ is $\sfW$-rational and $(\sfW, [\xi])$ is of type $\Ohm$. 
    
    Then $(\sfV = (\sfV_B, \sfV_\dR, \sfV_\et), [\eta_V])$ defined by \ref{thm: autVB exists} with $[\eta_V]$ be the tautological $\sfK$-level structure is an object of $\sM_{V, E}(\Sh_\sfK(G))$, and for each $(\sfW, [\xi]) \in \sM_{V, E'}(T)$, there exists a unique $\rho : T \to \Sh_\sfK(G)_{E'}$ such that $(\sfW, [\xi]) \iso \rho^* (\sfV, [\eta_V])$. 
\end{theorem}

Note that the objects in $\sM_{V, E'}(T)$ above are rigid, so $\sM_{V, E'}(T)$ is a set (cf. \ref{rmk: rigid}). 

\begin{proof}
    By applying \ref{thm: moduli over C}, we first construct a morphism $\rho_\IC : T_\IC \to \Sh_\sfK(G)_\IC$ such that there is an isomorphism $$\nu : \rho_\IC^*(\sfV, [\eta_V]) \sto (\sfW,[\xi])|_{T_\IC}.$$ 
    
    Next, we argue that $\rho_\IC$ descends over $E'$. Let $t \in T(\IC)$ be a point and let $x \in \Sh_\sfK(G)(\IC)$ be its image under $\rho_\IC$. It suffices to show that for any $\sigma \in \Aut(\IC/ E')$, $\rho_\IC(\sigma(t)) = \sigma(x)$. Let $(M, \gamma)$ be a test object at $t$ for $\sfW$ to be weakly AM (see \ref{def: compatible with AM}). Let $[\xi_M]$ be the $\sfK$-level structure on $\w_\et(M)$ obtained by pulling back $[\xi]_t$ along $\gamma_B \tensor \IA_f$. Then by \ref{sec: descent of one component}, there is an isomorphism $(\sfW_{\sigma(t)}, [\xi]_{\sigma(t)}) \iso (\w_\Hdg(M^\sigma), [\xi_M]^\sigma)$. Note that the fiber $\nu_t$ of $\nu$ at $t$ is an isomorphism $(\sfV_x, [\eta_V]_x) \sto (\sfW_t, [\xi]_t)$. Therefore, $(M, \nu_t^{-1} \circ \gamma)$ is a test object at $x$ for $\sfV$ to be weakly AM. Moreover, $[\xi_M]$ is the pullback of $[\eta_V]_x$ along $(\nu_{t}^{-1} \circ \gamma)_B \tensor \IA_f$. By \ref{sec: descent of one component} again, there is an isomorphism $(\sfV_{\sigma(x)}, [\eta_V]_{\sigma(x)}) \iso (\w_\Hdg(M^\sigma), [\xi_M]^\sigma)$. Therefore, we have shown that $(\sfV_{\sigma(x)}, [\eta_V]_{\sigma(x)}) \iso (\sfW_{\sigma(t)}, [\xi]_{\sigma(t)})$ as they are both isomorphic to $(\w_\Hdg(M^\sigma), [\xi_M]^\sigma)$. This implies by \ref{thm: moduli over C} that $\rho_\IC(\sigma(t)) = \sigma(x)$. 

    Now we have shown that $\rho_\IC$ descends to a morphism $\rho : T \to \Sh_\sfK(G)_{E'}$, and there is an isomorphism between $(\sfW, [\xi])$ and $\rho^*(\sfV, [\eta_V])$ when restricted to $T_\IC$. We conclude by the rigidity lemma \ref{lem: rigidity} that in fact we must have $(\sfW, [\xi]) \iso \rho^*(\sfV, [\eta_V])$. 
\end{proof}

\begin{remark}
    The above proof is certainly inspired by that of \cite[Cor.~5.4]{MPTate}, which deals with the special case when $(G, \Ohm)$ is an orthogonal type Shimura datum, $V$ is the standard representation of $G$, $T$ is the moduli space of quasi-polarized K3 surfaces of a certain degree, and $(\sfW, [\xi])$ is given by the primitive cohomology of the universal family. We remark though that after Madapusi-Pera proved the rationality of the peirod map, he proved in \cite[Prop.~5.6]{MPTate} that $\rho^* \sfV_\et \iso \sfW_\et$ by resorting to Deligne's big monodromy trick. It is in fact unnecessary if we only work with neat level structures, by the rigidity lemma \ref{lem: rigidity}. 
\end{remark}

\printbibliography

\end{document}